\documentclass[12pt]{amsart}
\usepackage{amscd}
\usepackage{amsmath}
\usepackage{verbatim}
\usepackage[normalem]{ulem}
\usepackage{longtable}
\usepackage{dynkin-diagrams}

\usepackage{epigraph}
\usepackage{amssymb}

\usepackage[all, cmtip,arrow]{xy}

\usepackage{pb-diagram,pb-xy}

\newcommand{\lradG}{0.25}
\newcommand{\darkradE}{0.115}

\numberwithin{equation}{section}

\newtheorem{thm}[equation]{Theorem}
\newtheorem{prop}[equation]{Proposition}

\newtheorem{lem}[equation]{Lemma}

\theoremstyle{definition}

\newtheorem{ntt}[equation]{}

\newcommand{\Br}{\mathop{\mathrm{Br}}}

\newcommand{\E}{\mathrm{E}}

\newcommand{\zz}{\mathbb{Z}}

\newcommand{\F}{\mathrm{F}}

\newcommand{\qq}{\mathbb{Q}}

\newcommand{\Aut}{\operatorname{Aut}}

\newcommand{\ZZ}{\mathrm{Z}}
\newcommand{\HH}{\mathrm{H}}

\title
[Tits construction and Rost invariant]
{Tits construction and Rost invariant}


\author
[Nikita Geldhauser]
{Nikita Geldhauser}

\author
[Victor Petrov]
{Victor Petrov}

\address{Geldhauser:
Mathematisches Institut der Universit\"at M\"unchen, Theresienstr. 39, D-80333 M\"unchen, Germany}
\email{geldhauser@math.lmu.de}

\address{Petrov: St.~Petersburg State University, 29B Line 14th (Vasilyevsky Island), 199178, St.~Petersburg, Russia\\
PDMI RAS, nab. Fontanki, 27, 191023, St.~Petersburg, Russia}
\email{victorapetrov@googlemail.com}

\thanks{The first author was supported by the DFG research grant SE 1721/4-1. The second author was supported (in part) by the BASIS foundation grant ``Young Russia Mathematics''}

\begin{document}

\maketitle

\epigraph{
Dedicated to Professor Nikolai Vavilov who taught us exceptional algebraic groups.\\
With gratitude, respect and grief.}

\begin{abstract}
We show a Springer type theorem for the variety of parabolic subgroups of type $1,2,6$ for all groups of type $\E_6$. As far as we know this gives the first example for the validity of the Springer theorem for projective homogeneous varieties of type $^2\E_6$ different from varieties of Borel subgroups. The proof combines several topics, notably the Rost invariant, a Tits construction, Cartan's symmetric spaces and indirectly the structure of the Chow motives of projective homogeneous varieties of exceptional type.
\end{abstract}

\newsavebox{\dEpic}
\savebox{\dEpic}(5,1){\begin{picture}(5, 1)
    \multiput(3,0.75)(1,0){2}{\circle*{\darkradE}}
    \multiput(3,0.25)(1,0){2}{\circle*{\darkradE}}
    \multiput(1,0.5)(1,0){2}{\circle*{\darkradE}}
    
    \put(1, 0.5){\line(1,0){1}}
    \put(3,0.75){\line(1,0){1}}
    \put(3, 0.25){\line(1,0){1}}
    
    \put(3,0.5){\oval(2,0.5)[l]}
    \end{picture}}

\section{Introduction}

In the present article we combine several ideas, notably the Rost invariant of algebraic groups and Cartan's symmetric spaces, to establish a theorem of  Springer type for some new classes of projective homogeneous varieties.

The celebrated classical Springer theorem asserts that an anisotropic quadratic form over a field remains anisotropic over an arbitrary field extension of the base field of odd degree. In other words, the existence of a zero-cycle of degree one on a smooth projective quadric implies that the quadric is isotropic.

Many articles are devoted to generalizations of the Springer theorem to other projective homogeneous varieties. For example, Bayer-Fluckiger and Lenstra showed the Springer theorem for all twisted forms of maximal orthogonal Grassmannians in \cite{BFL}. The Springer theorem is also known for the variety of Borel subgroups for groups of type $\F_4$ and $^1\E_6$ (a theorem of Rost), and for the variety of Borel subgroups for groups of type $^2\E_6$ (a theorem of Garibaldi \cite{G1}) and $\E_7$ (a theorem of Gille \cite[Theorem~C]{Gi1}). 

The most general known approach to prove results of Springer type relies on the norm principle of Gille \cite{Gi1} and its further refinements by Merkurjev \cite{Me}. For example, Geldhauser and Gille showed on this way in \cite{GiS} that the Springer theorem holds for the variety of maximal parabolic subgroups of type $7$ for all groups of type $\E_7$.

To stress the difficulty of the Springer theorem note that a theorem of Springer type for the varieties of Borel subgroups of groups of type $\E_8$
implies the Serre Conjecture II for fields of cohomological dimension $\le 2$. At the same time the Serre Conjecture II is considered as one of the major open problems in the theory of algebraic groups (see \cite{Gi2}).

We also refer to \cite{BQM} for related problems for algebras with involutions and to \cite{To} for related problems for groups of type $\E_8$.

In the present article we combine several ideas which are unrelated at the first glance and prove the Springer theorem for the variety of parabolic subgroups of type $1,2,6$ for groups of (inner and outer) type $\E_6$. As far as we know this gives the first example for the validity of the Springer theorem for projective homogeneous varieties of type $^2\E_6$ different from varieties of Borel subgroups.

\medskip

\textbf{Acknowledgements.} The authors would like to thank Skip Garibaldi and Philippe Gille for valuable discussions.

\section{Preliminaries on algebraic groups and symmetric spaces}

Let $G$ be a semisimple algebraic group over a field $F$ of characteristic different from $2$.

\begin{ntt}[Type of parabolic subgroups]
It is well known that there is an induced action of the absolute Galois group of $F$ on the Dynkin diagram of $G$, which
is called the $*$-action  (see \cite{KMRT}).

A {\it twisted flag variety} of $G$ over $F$ is the variety of parabolic subgroups of $G$ of some fixed type.
The type of a parabolic subgroup of $G$ corresponds to a subset of the set of vertices of the Dynkin diagram of $G$
invariant under the $*$-action, and we write $X_\Psi$ for the variety of parabolic subgroups of $G$ of type $\Psi$. In this article we use a complementary notation, 
e.g., we write $X_i$ for the variety of {\it maximal} parabolic subgroups of $G$ of type $i$. Under this convention the variery of Borel subgroups has type $1,2,\ldots,n$, where $n$ denotes the rank of $G$.

The enumeration of simple roots in this article follows Bourbaki.
\end{ntt}

\begin{ntt}[Tits algebras]

With every semisimple algebraic group $G$ over $F$ one can associate {\it Tits algebras}, which are the endomorphism rings of its irreducible representations (see \cite[\S27]{KMRT}). In general, they
are central simple algebras over finite separable extensions of $F$.
\end{ntt}

\begin{ntt}[Rost invariant and Tits index]

If $G$ an absolutely simple simply-connected algebraic group over $F$, then there is the {\it Rost invariant} $$R_G\colon \HH^1(-,G)\to \HH^3(-,\qq/\zz(2))$$ (see \cite[\S31B]{KMRT}, \cite{GMS}). Moreover, if the Tits algebras of $G$ are trivial (i.e., split), then there is a well-defined Rost invariant $r(G)\in \HH^3(F,\qq/\zz(2))$ of the group $G$ itself as explained in \cite[Lemma~2.1]{GP}.

Furthermore, with an embedding of absolutely simple simply-connected algebraic groups $G'\subset G$ one can associate a positive integer $m$ called the Rost multiplier (see \cite[Section~2.1]{G1}) such that the following diagram commutes

$$\xymatrix{
\HH^1(F,G') \ar[r]^-{R_S}\ar[d] & \HH^3(F,\qq/\zz(2)) \ar[d]^-{m\cdot} \\
           \HH^1(F,G)\ar[r]^-{R_G}  & \HH^3(F,\qq/\zz(2)).
           }$$
Note that the Rost multiplier is a combinatorial invariant and can be computed under the assumption that the base field is algebraically closed.

We will be interested in the embeddings of quasi-split simply-connected algebraic groups  ${{}^2\E_6\to\E_7}$ (see \cite[Proposition~3.6]{G1}). This embedding has Rost multiplier $1$. In particular, the respective twisted groups of type $^2\E_6$ and $\E_7$ with trivial Tits algebras have the same Rost invariant.

\medskip

With each semisimple algebraic group $G$ one can associate a {\it Tits index}, which consists of the Dynkin diagram of $G$, the $*$-action and a collection of circled vertices of the Dynkin diagram of $G$ (see \cite{Ti}). The Tits index encodes the maximal split torus of $G$ defined over $F$. If $X$ is a twisted flag $G$-variety of type $\Psi$, then $X$ is isotropic if and only if all vertices in $\Psi$ are circled on the Tits diagram of $G$.

If $G$ is an isotropic group of outer type $\E_6$ with trivial Tits algebras, then Garibaldi and Petersson describe its Rost invariant $r(G)$ in \cite[Proposition~2.3]{GP}. We represent their table below.

\begin{longtable}{cc} \\
Tits index & condition \\ \hline
quasi-split & $r(G) = 0$ \\
\parbox{5cm}{\begin{picture}(5, 1.1)
\put(0,0){\usebox{\dEpic}} 
\put(4,0.5){\oval(0.4,0.75)}
\put(1,0.5){\circle{\lradG}}
\end{picture}}& $r(G)$ is a nonzero pure symbol divisible by $K$\\
\parbox{5cm}{\begin{picture}(5, 1.1)
\put(0,0){\usebox{\dEpic}} 
\put(4,0.5){\oval(0.4,0.75)}
\end{picture}}& $r(G)$ is a symbol not divisible by $K$\\
\parbox{5cm}{\begin{picture}(5, 1)
\put(0,0){\usebox{\dEpic}} 
\put(1,0.5){\circle{\lradG}}
\end{picture}}& $r(G)$ is not a pure symbol\\
\parbox{5cm}{\begin{picture}(5, 1)
\put(0,0){\usebox{\dEpic}} 
\put(1,0.5){\circle{\lradG}}
\put(2,0.5){\circle{\lradG}}
\end{picture}} & impossible diagram, if the Tits algebras of $G$ are split\\
\caption{Rost invariant of isotropic groups of type $\E_6$}
\end{longtable}

Note that in the case $K=F\times F$, the respective group $G$ of type $\E_6$ is of inner type and there are two cases of possible Tits indices for isotropic $G$: $r(G)=0$ if and only if $G$ is split and $r(G)$ is a nonzero symbol if and only if the circled vertices of the Tits index are exactly $1$ and $6$.

\medskip

In the present article we will also essentially use the following result about the Rost invariant of groups of type $\E_7$ (see \cite[Theorem~10.10]{GPS}). Let $G$ be a group of type $\E_7$ with trivial Tits algebras and let $X$ be the variety of parabolic subgroups of $G$ of type $7$. Then the variety $X$ has a rational point if and only if $6r(G)=0$ and $3r(G)$ is a pure symbol in $\HH^3(F,\zz/2\zz)$

The proof of this result in \cite{GPS} is a combination of \cite[Corollary~3.5]{GiS} (the proof of this Corollary relies on the norm principle of Gille--Merkurjev) and of a computation of the Chow motive of the variety $X$. The latter computation relies, in particular, on the binary motive theorem, and, thus, indirectly on the structure of the Steenrod algebra.

\medskip

Finally, we will use the following result of Rost \cite[Proposition~2]{Ro}: If $\alpha\in\HH^n(F,\zz/2\zz)$ is equal to a pure symbol over some odd degree field extension of the field $F$, then $\alpha$ can be represented by a pure symbol already over $F$.
\end{ntt}

\begin{ntt}[Tits construction]
The celebrated Freudenthal--Tits magic square unites constructions of exceptional algebraic groups using two algebras as an input (see \cite[p.~540]{KMRT}). In the present article we will be interested in a construction of groups $G(A,K)$ of outer type $\E_6$ 
based on an Albert algebra $A$ and a quadratic extension $K/F$ (a $\mu_2\times \F_4$-construction). We refer to \cite{GP} for the details and properties of this construction. Note also that the Tits algebras of all groups $G(A,K)$ arising from this Tits construction are trivial.

It is well known that the Albert algebras $A$ have three cohomological invariants
$${f_3(A)\in\HH^3(F,\zz/2\zz)},\  {g_3(A)\in\HH^3(F,\zz/3\zz)}, 
\text{ and }{f_5(A)\in\HH^5(F,\zz/2\zz)}.$$ Moreover, the invariants $f_3$ and $g_3$ are the modulo $2$ and modulo $3$ components of the Rost invariant of the respective group $G(A,K)$.

Note, however, that the order of the Rost invariant for general groups of type $^2\E_6$ (i.e., of groups which do not arise from this Tits construction) and with trivial Tits algebras is $12$ and not $6$.
\end{ntt}

\begin{ntt}[Symmetric spaces]
In the present article we use one more tool, namely, Cartan's symmetric spaces.

Recall that a symmetric space is a homogeneous variety $X=G/H$, where $G$ is a reductive algebraic group and $H$ is an open subgroup of the fixed point
group of some involution $\sigma$ of $G$.

A split torus ${\mathbb G}_m$ is called $\sigma$-split, if it is $\sigma$-stable and $\sigma$ acts on it by the rule $\sigma(t)=t^{-1}$. The vertices of the Dynkin diagram corresponding to maximal parabolic subgroups whose Levi subgroups are centralizers of $\sigma$-split tori are colored in white in the \emph{Satake diagram} of $G/H$ (and the remaining vertices are colored in black). Note that Satake diagrams also contain a Galois action, but in the present article we will only use symmetric spaces EIV and EVII, where the Galois action is trivial.  

The white vertices of Satake diagrams correspond to possible types of maximal parabolic subgroups $P$ such that $\sigma(P)$ is opposite to $P$, see \cite[Lemma~2.9 and 2.11]{S}. In this case such parabolic subgroups form an open orbit on $G/P$ under the action of $H$.

The stabilizer of $P$ in $H$ can be computed as follows: one removes the white vertex corresponding to $P$ from the Satake diagram of $G$ and looks at the respective new Satake diagram in the table of symmetric spaces. It will correspond to some $G'/H'$. Then $H'$ is the connected component of the stabilizer we are looking for. Note, however, that since $P\cap\sigma(P)$ is reductive and not semisimple, an additional discrete factor may appear.

As an example we consider the symmetric space of type EVII and the vertex $i=7$ on the Dynkin diagram of type $\E_7$.

We reproduce a relevant part of the table of symmetric spaces, where $\E_6^{red}$ stands for a non-simple reductive group of type $\E_6$.

\begin{longtable}{l|l|l||l} 
Cartan type & Type of $G$ & Type of the connected & Satake diagram of $G/H$ \\
&&component of $H$ &\\
\hline
EIV &$\mathrm{E}_{6}$ &  $\mathrm {F}_4$ & 	\dynkin[edge length=.5cm] E{IV}\\
EVII& $\mathrm{E}_{7}$ & $\E_6^{red}$  & 	\dynkin[edge length=.5cm] E{VII}\\
\caption{Cartan's symmetric spaces $G/H$ and Satake diagrams}
\end{longtable}

Removing vertex $7$ for type EVII we obtain the Satake diagram of Cartan type EIV. In particular, it follows that the connected component of the stabilizer for the open orbit of the action of $\E_6^{red}$ on the $27$-dimensional variety $\E_7/P$, where $P$ stands for a parabolic subgroup of type $7$, is of Dynkin type $\F_4$.

\end{ntt}

\section{Springer theorem for groups of type $\E_6$}

Consider the split adjoint group $G$ of type $\E_7$ over a field $F$ of characteristic different from $2$. Let us look at a parabolic subgroup $P$ of type $7$. The center of its Levi subgroup is ${\mathbb G}_m$ and contains a unique nontrivial involution $\sigma$. Now $H=G^\sigma$ is the normalizer of the Levi subgroup (which is the connected component $H^\circ$) and is a semidirect product of $H^\circ$ and $\zz/2\zz$. One says that $G/H$ is a symmetric space of type EVII.

The nontrivial element $\tau$ in $\zz/2\zz$ is an involution on $G$ commuting with $\sigma$. One can show that $\tau$ and $\sigma$ are conjugate in $G$, but we will not use this fact.

Now $G/H$ is a symmetric space and, in particular, the parabolic subgroup $P$ acts on it with a finite number of orbits. Then, vice versa, $H$ acts on $G/P$ with the same number of orbits and, in particular, there is an open orbit for this action. The open orbit consists of all parabolic subgroups $Q$ of type $7$ such that $\sigma(Q)$ is opposite to $Q$. One can choose for $Q$ any (of two) parabolic subgroups containing $(G^\tau)^\circ$. Then the stabilizer $S$ of this point in $H$ equals the subgroup of $H^\tau$ generated by a split subgroup of type $\F_4$ and $\tau$, while $H^\tau$ itself is the product $S$ and $\sigma$.

We summarize the groups above in a table below. We write $L$ for the Levi subgroup and $P$ for a parabolic subgroup of $G$ of type $7$.

\begin{longtable}{l|l} 
Group & Type  \\
\hline
$G$ &$\mathrm{E}_{7}$ \\
$H=G^\sigma=N_G(L(P))$ & $\mathrm{E}_{6}^{red}\rtimes\zz/2\zz$ \\
$H^\circ=(G^\sigma)^\circ=L(P)$ & $\mathrm{E}_{6}^{red}$ \\
$H^\tau$ & $\F_4\times \zz/2\zz\times \zz/2\zz$\\
$(H^\tau)^\circ$ & $\F_4$\\
$G^\tau$ & $\mathrm{E}_{6}^{red}\rtimes\zz/2\zz$\\
$(G^\tau)^\circ$ & $\mathrm{E}_{6}^{red}$\\
$S$ & $\F_4\times \zz/2\zz$\\
\caption{Summary of groups}
\end{longtable}

\begin{prop}\label{prop:main}
The image of a class $[\xi]\in\HH^1(F,H)$ in $\HH^1(F,G)$ produces an isotropic group of type $\E_7$ possessing a parabolic subgroup $P$ of type $7$ defined over $F$ if and only if $[\xi]$ comes from some $[\zeta]\in\HH^1(F,S)$.
\end{prop}
\begin{proof}
We may assume that $F$ is infinite. As explained above, the twisted form ${}_{\xi}(G/P_7)$ is a smooth compactification of the open orbit $U={}_{\xi}(H/S)$. It has a rational point if and only if ${}_{\xi}G$ possesses a parabolic subgroup $P$ of type $7$ defined over $F$. In this case the unipotent radical of the opposite parabolic subgroup defines an open subvariety of ${}_{\xi}(G/P_7)$ isomorphic to the affine space ${\mathbb A}_F^{27}$. Since $F$ is infinite, there is a rational point in ${\mathbb A}_F^{27}\cap U$. It remains to apply \cite[Proposition~37]{Se}.
\end{proof}

Now we give an interpretation of this result in terms of a construction by Jacques Tits. Given an Albert algebra $A$ over $F$ and a quadratic field extension $K/F$ one can construct a simply connected simple group $G(A,K)$ of outer type $\E_6$ (see \cite{GP}).

We denote by $\E_6^{sc}$ (resp. $\E_6^{ad}$) a split simply connected (resp. adjoint) simple group of Dynkin type $\E_6$. All simply connected groups of type $\E_6$ over $F$ are parameterized by cocycles with values in $\Aut(\E_6^{sc})\simeq \E_6^{ad}\rtimes\zz/2\zz$. The latter group is isomorphic to $H$ modulo the center of $H^\circ$, which is ${\mathbb G}_m$.

We have the following commutative diagram
$$\xymatrix{
\HH^1(F,S) \ar[r]\ar[d]^-{=} & \HH^1(F,H) \ar[d] \\
           \HH^1(F,S)\ar[r]  & \HH^1(F,\E_6^{ad}\rtimes\zz/2\zz).
           }$$

\begin{lem}
A class $[\theta]\in\HH^1(F,\E_6^{ad}\rtimes\zz/2\zz)$ comes from $\HH^1(F,H)$ if and only if the Tits algebras are trivial. Moreover, it comes from $\HH^1(F,S)$ if and only if it is isomorphic to the Tits construction $G(A,K)$ for some $A$ and $K$.
\end{lem}
\begin{proof}
Note that $[\theta]$ defines a group of type $\E_6$ and every group of type $\E_6$ over $F$ is equipped with a quadratic extension $K/F$.

Twisting the exact sequence
$$
\xymatrix{
1\ar[r] & {\mathbb G}_m\ar[r] & H\ar[r] & \E_6^{ad}\rtimes\zz/2\zz\ar[r]&1
}
$$
by the element in $\HH^1(F,\zz/2\zz)$ corresponding to $K$ one obtains
$$
\xymatrix{
1\ar[r] & R_{K/F}^{(1)}({\mathbb G}_m)\ar[r] & \E_6^{qs,red}\rtimes\zz/2\zz\ar[r] & \E_6^{qs,ad}\rtimes\zz/2\zz\ar[r]&1,
}
$$
where $qs$ means a quasi-split group and $R_{K/F}^{(1)}$ stands for the kernel of the norm map.

The class $[\theta']$ corresponding to $[\theta]$ belongs actually to $\HH^1(F,\E_6^{qs,ad})$, and the connecting map sends it to the element in $\HH^2(F,R_{K/F}^{(1)}({\mathbb G}_m))\le\Br(K)$ corresponding to the class of the Tits algebra, and the first claim follows.

Consider the split Albert algebra $A_0$ and the trivial \'etale extension $F\times F$. Then $G(A_0,F\times F)$ is a split group of type $\E_6$. Now twisting by 
$$\ZZ^1(F,\Aut(A_0)\times\Aut(F\times F))=\ZZ^1(F,S)$$
produces $G(A,K)$, and each $G(A,K)$ can be obtained in this way.
\end{proof}

\begin{thm}\label{thm:main}
Consider a simply connected group ${\mathcal G}$ of type $\E_6$ over $F$ with trivial Tits algebras. Then ${\mathcal G}$ can be obtained by the Tits construction if and only if $3r({\mathcal G})$ is a pure symbol in $\HH^3(F,\zz/2\zz)$.
\end{thm}
\begin{proof}
Recall that above we denoted by $G$ a split adjoint group of type $\E_7$.

Let $[\theta]\in\HH^1(F,\E_6^{ad}\rtimes\zz/2\zz)$ be a class defining ${\mathcal G}$. Since the Tits algebras are trivial, it can be lifted to a class $[\xi]\in\HH^1(F,H)$ uniquely. Denote by $[\phi]$ the image of $[\xi]$ in $\HH^1(F,G)$. Then the Rost invariants of $[\phi]$ and $\mathcal G$ are equal.

By \cite[Theorem~10.10]{GPS} ${}_\phi(G/P)$, where $P$ is a parabolic subgroup of $G$ of type $7$, has a rational point if and only if $3r({\mathcal G})$ is a pure symbol in $\HH^3(F,\zz/2\zz)$. Now the claim follows from Proposition~\ref{prop:main}.
\end{proof}

As an application we prove a version of the Springer theorem.

\begin{thm}\label{thm:Sp}
Let $[\theta]$ be a class in $\HH^1(F,\E_6^{ad}\rtimes\zz/2\zz)$ and $X={}_\theta(\E_6/P_{1,2,6})$ be the corresponding twisted flag variety of parabolic subgroups of type $1,2,6$. Assume that $X$ has a rational point over finite field extensions $F_i/F$ with $\gcd([F_i:F])=1$. Then $X$ has a rational point over $F$.
\end{thm}
\begin{proof}
First note that the Tits algebras of the respective group of type $\E_6$ are trivial over all $F_i$ and so by the restriction-corestriction argument
they are trivial over $F$.

Consider the group ${\mathcal G}$ corresponding to $\theta$ and look at its Rost invariant $r({\mathcal G})$. Using restriction-corestriction argument again we see that $2r({\mathcal G})=0$. Moreover, by \cite[Proposition~2.3]{GP} over each field $F_i$ the respective Rost invariant $r_{F_i}({\mathcal G})$ is a pure symbol in $\HH^3(F,\zz/2\zz)$. Since there exists $F_i$ with an odd degree $[F_i:F]$, by \cite[Proposition~2]{Ro} $r({\mathcal G})$ is a pure symbol in $\HH^3(F,\zz/2\zz)$ as well.

This symbol splits over $K$ by \cite[Proposition~2.3]{GP} and by the restriction-corestriction argument. By Theorem~\ref{thm:main} ${\mathcal G}$ can be obtained by the Tits construction from some Albert algebra $A$. We have $f_3(A)=r({\mathcal G})$ and $g_3(A)=0$. By \cite[Proposition~2.3]{GP} it remains to show that ${\mathcal G}$ is isotropic. By \cite[Theorem~0.2]{GP} it suffices to show that $A$ contains nonzero nilpotent elements, i.e., that the invariant $f_5(A)=0$. But again this can be deduced
using the restriction-corestriction argument.
\end{proof}

\end{document}